\documentclass[oneside,reqno,11pt]{amsart}

\usepackage{mathtools}	
\usepackage{amsmath}
\usepackage{amsthm}
\usepackage{bbm}
\usepackage{amssymb}
\usepackage{marginnote}
\usepackage{pstricks}
\usepackage[percent]{overpic}
\usepackage{graphicx} 
\newcommand{\sgn}{\operatorname{sgn}}

      \theoremstyle{plain}
      \newtheorem{theorem}{Theorem}[section]
      \newtheorem{lemma}[theorem]{Lemma}

      \newtheorem{assumption}[theorem]{Assumption}       
      \theoremstyle{definition}
      \newtheorem{definition}[theorem]{Definition}      
      \theoremstyle{remark}
      \newtheorem{remark}[theorem]{Remark}
      
\begin{document}
\bigskip

\title[Spacings -- an example for Universality in Random Matrix Theory]{Spacings -- an example for Universality in Random Matrix Theory}
   
   \author{Thomas Kriecherbauer}
   \address{Inst. for Mathematics, Univ. Bayreuth, 95440 Bayreuth, Germany}
   \email{thomas.kriecherbauer@uni-bayreuth.de}

    \author{Kristina Schubert}
   \address{Inst. for Math. Stat., Univ. M\"unster, Orl\'{e}ans-Ring 10, 48149 M\"unster, Germany}
   \email{kristina.schubert@uni-muenster.de}



\begin{abstract}
Universality of local eigenvalue statistics is one of the most striking phenomena of Random Matrix Theory, that also accounts for a lot of the attention that the field has attracted over the past 15 years. In this paper we focus on the empirical spacing distribution and its Kolmogorov distance from the universal limit. We describe new results, some analytical, some numerical, that are contained in \cite{Diss}. A large part of the paper is devoted to explain basic definitions and facts of Random Matrix Theory, culminating in a sketch of the proof of a weak version of convergence for the empirical spacing distribution $\sigma_N$ (see (\ref{sec5:eq1})).
\end{abstract}

 \maketitle

\section{Introduction}
\label{sec:1}
The roots of the theory of random matrices reach back more than a century. They can be found, for example, in the study of the Haar measure on classical 
groups \cite{THur} and in statistics \cite{TWish}. The field experienced a first boost in the 1950's due to a remarkable idea of E.~Wigner. He suggested to model the statistics of highly excited energy levels of heavy nuclei by the spectrum of random matrices. Arguably the most striking aspect of his investigations was how well the random eigenvalues described the distribution of spacings between neighbouring energy levels. Even more surprising were the subsequent discoveries that the eigenvalue spacing distributions are also relevant in a number of different areas of physics (e.g.~as a signature for quantum chaos) and somewhat exotically also in number theory for the description of zeros of zeta functions (see \cite[Ch.~2 and Part III]{akemann2011oxford} for recent reviews). Due to these developments Random Matrix Theory became an active area of research that was prospering for many years mainly in the realm of physics. It was only about 15 years ago that random matrices started to attract broader interest also within mathematics 
stretching over a variety of different areas. The reason for this second boost was the discovery 
\cite{TBaDeJo}
that after appropriate rescaling the length of the longest increasing subsequence of a random permutation on $N$ letters displays for large $N$ the same fluctuations as the largest eigenvalue of a $N \times N$ random matrix (from a particular set of ensembles). Again it turned out that the distribution of the largest eigenvalue (in the limit $N \to \infty$) defines a fundamental distribution that comes up in a number of seemingly unrelated models of combinatorics and statistical mechanics (e.g.~growth models, interacting particle systems; see \cite{Handbook_Spohn} and 
\cite{TKrKr}
for recent reviews).  

In summary, we have seen that the statistics of eigenvalues of random matrices display a certain degree of universality by describing the fluctuations in a varied list of stochastic, combinatorial and even deterministic (zeros of zeta functions) settings. In this paper, however,  we will be concerned with a second aspect of universality that is known as the universality conjecture in random matrix theory. It states that in the limit of large matrix dimensions local eigenvalue statistics (see the beginning of section \ref{Tsec:3} for an explanation of the meaning of this term) only depend on the symmetry class (cf.~sect.~\ref{Tsec:2}) of the matrix ensemble but not on other details of the probability measure. We will discuss this conjecture in the context of the nearest neighbour spacing distribution that has received much less attention in the literature than other statistical quantities such as $k$-point correlations or gap probabilities. We focus on the question of convergence of the empirical spacing distribution of eigenvalues. 

Besides the standard monograph \cite{mehta}, a number of books have appeared recently \cite{deift1}, \cite{deift2}, \cite{Forrester_log_gases}, \cite{anderson}, \cite{Tao_book}, which present nice introductions into various aspects of Random Matrix Theory.  An impressive collection of topics from  Random Matrix Theory and its applications can be found in \cite{akemann2011oxford}. However, in all these books the information on the convergence of the empirical spacing distribution is somewhat sparse, except for \cite{deift1} and \cite{katzsarnak} in the case of unitary ensembles ($\beta=2$). It is one of the goals of this paper to give a concise and largely self-contained update of \cite{deift1}, \cite{katzsarnak} w.r.t.~spacing distributions including also orthogonal ($\beta=1$) and symplectic ($\beta=4$) ensembles. 

The paper is organised as follows.
First, we introduce in section \ref{Tsec:2} three important types of matrix ensembles that generalize the classical Gaussian ensembles. 
In order to define our prime object of study, the empirical spacing distribution (section \ref{sec:3_2}), we first discuss the spectral limiting density for all three types of ensembles in section \ref{sec:3_1}. From section  \ref{Tsec:4} to  \ref{sec:result} we only treat invariant ensembles. We first state how $k$-point correlations are related to orthogonal polynomials and recall what is known 
about their convergence (section  \ref{Tsec:4}). These results are used in section \ref{sec:PointwiseConvergence} to sketch the proof of a first convergence result (\ref{sec5:eq1}) for the empirical spacing distribution. Our main new result Theorem \ref{main_theo}, that is proved in \cite{Diss}, is stated in section \ref{sec:result} together with related results for circular ensembles from the literature. We close by mentioning numerical results of \cite{Diss}. They indicate that a version of the Central Limit Theorem, similar to the one proved in  \cite{soshnikov}  for COE and CUE, should also hold for the ensembles discussed in this paper. 

\section{Random matrix ensembles}
\label{Tsec:2}

Starting with the classical Gaussian ensembles we introduce in the present section three different types of generalisations, Wigner ensembles, Invariant ensembles, and $\beta$-ensembles. Together they constitute a large part of the ensembles studied in Random Matrix Theory. Some references are provided where the reader can learn about the main techniques to analyse these ensembles. The concept of symmetry classes is briefly discussed.

We begin by defining one of the most prominent matrix ensembles, the Gaussian Unitary Ensemble (GUE).
GUE is a collection of probability measures on Hermitian $N \times N$ matrices $X$, $N \in
\mathbb N$, where the diagonal entries $x_{jj}$ and the real and imaginary parts of the upper triangular entries
$x_{jk} = u_{jk} + i v_{jk}$, $j<k$ are all independent and normally distributed with $x_{jj} \sim {\mathcal N}(0, 1/\sqrt{N})$, 
$u_{jk}$, $v_{jk} \sim {\mathcal N}(0, 1/\sqrt{2N})$. GUE has the following useful properties.
\begin{itemize}
 \item[1.] The entries are independent as far as the Hermitian symmetry permits.
 \item[2.] The probability measure is invariant under conjugation by matrices of the unitary group, i.e.~under change of 
    orthonormal bases. In fact, this explains why the ensemble is called ``unitary'' (and the reference to Gauss is 
    due to the normal distribution).   
    Moreover, one can compute the joint distribution of the eigenvalues explicitly. The vector of 
    eigenvalues $(\lambda_1, \ldots, \lambda_N)$ is distributed on $\mathbb R^N$ with Lebesgue-density
      \begin{equation} \label{T2.5}
         Z_{N, \beta}^{-1} \prod_{j<k} |\lambda_k - \lambda_j|^\beta \prod_{i} e^{-\beta N \lambda_i^2/4} \ d\lambda_i , \quad \beta =2,
      \end{equation}
    where $Z_{N, \beta}$ denotes some norming constant.
\end{itemize}
Each of these properties comes with a set of techniques to analyse statistics of eigenvalues. In turn, these techniques 
can be applied to a large number of matrix ensembles that share this particular property. More precisely:
\begin{itemize}
 \item[1.] {\em Wigner ensembles} have independent entries as far as the symmetry of the matrix permits. 
The distributions of the entries do not need to be normal or identical, but must satisfy some conditions on the moments.
Except for the Gaussian case, Wigner ensembles are not unitarily invariant and the joint distribution of eigenvalues is generally not known.
Many results for such ensembles (e.g.~Wigner semi circle law, distribution of the largest eigenvalue) can be obtained 
via the {\em method of moments}, i.e.~by analysing the moments of the empirical measure of the eigenvalues (see e.g.~\cite{hiaipetz}, cf.~sect.~\ref{Tsec:3}). More recently, 
very powerful new techniques have been introduced by Erd\"os et al and independently  by Tao and Vu (see e.g.~\cite{Erdos2, TaoVu} and references therein).

\item[2.] {\em Invariant ensembles} keep the property of invariance under conjugation by unitary matrices. The ensembles considered in this
class all have in common that the joint distribution of eigenvalues is given by a measure of the form 
\begin{equation} \label{T2.10}
  Z_{N, \beta}^{-1} \prod_{j<k} |\lambda_k - \lambda_j|^\beta \prod_{i} d\mu_N(\lambda_i) , \quad \beta = 2
\end{equation}
where $\mu_N$ denotes some (positive) finite measure on $\mathbb R$ with sufficient decay at infinity to guarantee finiteness 
of the measure on $\mathbb R^N$. As we  explain in section~\ref{Tsec:4} it is exactly the structure of (\ref{T2.10}),
i.e.~a product measure with dependencies introduced by the square of the Vandermonde determinant, for which 
the {\em method of orthogonal polynomials} can be applied. Note that 
such measures, with $\mu_N$ supported on discrete sets, were also central for proving the appearance of local eigenvalue statistics in some
of the models from statistical mechanics described in the Introduction
(see e.g.~\cite{TKrKr} for an elementary exposition in the case of interacting particle systems).
\end{itemize}

Using these two types of generalisations of GUE we may already generate a great number of matrix ensembles. These consist of Hermitian matrices only and we say that they belong to the same symmetry class. By the universality conjecture we expect that in the limit $N \to \infty$ all these ensembles display the same local spectral statistics. 

If one replaces in the definition of GUE above the Hermitian matrices by real symmetric resp.~by quaternion self-dual matrices, keeping the independence of the entries as well as their normal distributions (with appropriately chosen variances), one obtains the Gaussian Orthogonal Ensemble (GOE) resp.~the Gaussian Symplectic Ensemble (GSE). These ensembles can be generalised as above, yielding again Wigner ensembles or invariant ensembles and the only difference compared to the discussion above is that in (\ref{T2.5}), (\ref{T2.10}) we have to choose $\beta = 1$ resp.~$\beta = 4$. In this way we have introduced two more symmetry classes which then together constitute all classes from Dyson's threefold way. As it was discovered some 30 years after Dyson's classification result from 1962, it is useful and natural to enlarge this list to a grand total of 10 symmetry classes, thus providing a significant increase in applications of Random Matrix Theory in physics, statistics and mathematics alike (see \cite{Handbook_Zirnbauer} for a recent survey). It should be noted that from the perspective of invariant matrix ensembles the resulting joint distributions of the eigenvalues are of the form (\ref{T2.10})
with $\beta \in  \{ 1, 2, 4\}$ for all ten symmetry classes. As we will argue in section \ref{sec:result} there exists a large class of invariant ensembles from all 10 symmetry classes for which the localised and appropriately rescaled empirical spacing distributions (see section \ref{Tsec:3}) converge to universal limits that only depend on the value of $\beta$. The {\em method of orthogonal polynomials} mentioned above can also be applied for $\beta = 1$, $4$. However, it is more technical and its range of applicability is less general than in the case $\beta =2$, see e.g.~\cite{deift2}.

There is a third property of Gaussian ensembles that leads to a different type of generalisation. The basic observation is the following. If one applies the
Householder transformation to GOE in a suitable way, one obtains probability measures on $N\times N$ Jacobi matrices 
 (i.e.~real symmetric, tridiagonal matrices with positive off-diagonal entries). By construction they induce the same joint distributions
of eigenvalues as (\ref{T2.5}) with $\beta = 1$. For this ensemble, the entries are again independent (as symmetry permits) with normal 
distributions on the diagonal and some $\chi$-distributions for the off-diagonal entries.
General {\em $\beta$-ensembles} are now generated by modifying the variances of the $\chi$-distributions on the off-diagonals. For any $\beta > 0$ this can be done in such a way that the joint distribution of eigenvalues is given
by (\ref{T2.5}) with the prescribed value of $\beta$. A key insight into the analysis of these ensembles is that for 
large matrix dimensions eigenvalues of the Jacobi matrices may be approximated by the spectrum 
of a specific stochastic Schr\"odinger operator, see e.g.~\cite{TRaRiVi}. 
Note that the local eigenvalue statistics of $\beta$-ensembles are different for each value of $\beta$. Obviously, they reduce to the classical Gaussian ensembles if and only if $\beta \in  \{ 1, 2, 4\}$.

\section{The empirical spacing distribution - localised and rescaled}
\label{Tsec:3}

In this section we define the empirical spacing distribution as one prime example for local eigenvalue statistics. By the latter we mean, firstly,
that the spectrum is localised by considering only some part of the spectrum and, secondly, that the spectrum is being rescaled such that the average distance between neighbouring eigenvalues is constant and of order 1 in the considered spectral region. In order to perform such operations we must first understand the {\em limiting spectral density} of the ensemble.

\subsection{The limiting spectral density}\label{sec:3_1}
We denote the ordered eigenvalues of a matrix $H$ from one of the ensembles described in section \ref{Tsec:2} by $\lambda_1^{(N)}(H) \le \lambda_2^{(N)}(H) \le \ldots \le \lambda_N^{(N)}(H)$. The corresponding $N$-tuple $\lambda^{(N)}(H)$ thus defines a point in the Weyl chamber that we denote by 
$\mathcal{W}_N \coloneqq \{x \in \mathbb R^N: \,  x_1\leq \ldots \leq x_N\}$. Moreover, we abbreviate $\lambda_j^{(N)}(H)$ by $\lambda_j$ from now on to keep the notation manageable. 

We associate to each $\lambda \in \mathcal{W}_N$ its counting measure $\delta_{\lambda} \coloneqq \frac{1}{N} \sum_{j=1}^N \delta_{\lambda_j}$ which defines a probability measure on $\mathbb R$. By the {\em limiting spectral density} we mean a function  $\psi: \mathbb R \to [0, \infty )$ satisfying for all $s \in \mathbb R$ that
\begin{align*}
\mathbb E_{N,\beta} 
\left(  \int_{-\infty}^{s} \, d\delta_{\lambda} \right) \to \int_{-\infty}^s \psi(t) \, dt \quad \text{ as } N \to \infty \, .
\end{align*}
It is known for ample classes of both Wigner ensembles and invariant ensembles as well as for $\beta$-ensembles that the spectral density exists. For Wigner ensembles one can show mainly by combinatorial methods that on average the moments of $\delta_{\lambda}$ converge to the moments of the semi-circle distribution (Wigner semi-circle law). The first steps of the proof are provided by the simple observation that for $k \in \mathbb N$ one has
\begin{align*}
\mathbb E_{N,\beta} 
\left(  \int_{\mathbb R}  t^k \, d\delta_{\lambda} (t) \right) = \frac{1}{N}\mathbb E_{N,\beta} (\mbox{tr }(H^k) )
\end{align*}
together with an expansion of the right hand side as a sum of expectations of products of entries of $H$ that can be simplified by using the independence of the entries (method of moments, see e.g.~\cite{hiaipetz}). 

Next we turn to $\beta$-ensembles. Here the limiting spectral density is again given by the Wigner semi-circle law. The proof, however, follows a different path. Recall that the joint distribution of eigenvalues is given by (\ref{T2.5}). Its density can therefore be rewritten in the form
\begin{align}
\label{T3.5}
Z_{N, \beta}^{-1} \exp\left[- \beta N^2 I(\delta_{\lambda}) \right] \, \, \, \mbox{with} \, \, \, 
I(\nu) := - \frac{1}{2} \int_{x \ne y} \!\!\!\!\!\!\!\!\log | x-y | d\nu(x) d\nu(y) + \int \frac{x^2}{4} d\nu(x) \, .
\end{align}
We may think of $I$ as a functional defined on all probability measures on $\mathbb R$. It is a well known fact in logarithmic potential theory that $I$ has a unique minimizer that is given by the semi-circle law. Since we have the factor $N^2$ in the exponent in (\ref{T3.5}) it is intuitively clear that for large $N$ only those vectors $\lambda$ will be relevant for which the corresponding counting measure $\delta_{\lambda}$ is close to the minimizer of $I$. This idea can be used to prove the Wigner semi-circle law for $\beta$-ensembles. Moreover, this idea can also be applied to prove the existence of the limiting spectral density for a large class of invariant ensembles (see e.g.
\cite{TJo98}). Indeed, let us assume that in (\ref{T2.10}) the measure $d\mu_N$ has a Lebesgue-density of the form
\begin{align}\label{T3.15}
d\mu_N(x) = e^{-N V(x)}dx \quad \mbox{satisfying} \quad \lim_{|x| \to \infty} \, \, \frac{V(x)}{\log |x|} = \infty \, , 
\end{align}
in order to guarantee that the measure (\ref{T2.10}) is finite. Under mild regularity assumptions on $V$ one can proof that the functional
\begin{equation}\label{midnight}
I_V(\nu) \coloneqq - \frac{1}{2} \int_{x \ne y} \!\!\!\!\log | x-y | d\nu(x) d\nu(y) + \int V(x) d\nu(x) \, ,
\end{equation}
defined on the probability measures on $\mathbb R$ has an unique minimizer $\nu_V$ with a Lebesgue-density $\psi = \psi_V$. As argued above, one can show that $\psi$ is the limiting spectral density of the ensemble (see e.g.~\cite[chap. 6]{deift1} for an elementary exposition). In the literature on invariant ensembles one also finds a slightly more general setting where in the formula (\ref{T3.15}) for the density of $d\mu_N$ the function $V$ is replaced by $N$-dependent functions $V_N$ that converge to some function $V$ satisfying the growth condition (\ref{T3.15}).

Note, that for invariant ensembles the limiting spectral density depends on $V$ and is therefore not an universal quantity. This is not a contradiction to the universality conjecture of Dyson since the limiting spectral density is a global quantity whereas the universality conjecture only refers to local eigenvalue statistics. 

\subsection{The empirical spacing distribution}\label{sec:3_2}
We use the limiting spectral density in order to rescale the eigenvalues. Let $a$ denote a point in the interior of the support of $\psi$ where the limiting density is positive, i.e.~$\psi(a)>0$. We assume further that $a$ is a point of continuity for $\psi$.  For eigenvalues $\lambda_i$ that are close to $a$ the expected distance of neighbouring eigenvalues is given to leading order by $(N \psi(a))^{-1}$. Therefore we introduce the rescaled and centred eigenvalues
\begin{align}\label{T3.20}
		\widetilde{\lambda}_i\coloneqq  ( \lambda_i -a) N\psi(a).
	\end{align}
Considering only eigenvalues $\lambda_i$ that lie in an ($N$-dependent) interval $I_N$ that is centred at $a$ and has vanishing length
$|I_N| \to 0$ for $N \to \infty$, we expect that their rescaled versions $\widetilde{\lambda}_i$ have a spacing that is close to $1$ on average.
We introduce
	\begin{equation*} 
		A_N\coloneqq  N \psi(a)(I_N-a)=\{N \psi(a)(t-a) \mid t \in I_N \}
	\end{equation*}
and observe that $\lambda_i \in I_N$ if and only if $\widetilde{\lambda}_i \in A_N$. Therefore and by the expected unit spacing of the rescaled eigenvalues  we conclude that the length of $A_N$ gives the average of the number of eigenvalues $\lambda_i$ that lie in $I_N$ to leading order. For our considerations we assume that this number and hence $ N |I_N|= |A_N|/\psi(a)$ tends to infinity for $N \to \infty$. We summarize our assumptions on the length of $I_N$.
	\begin{equation}\label{T3.25}
	|I_N| \to 0 \, , \quad \quad  N |I_N|  \to \infty \quad \quad \text{ for } N \to \infty.
	\end{equation}
Finally, we define our main object of interest, the empirical spacing distribution. As above we denote the eigenvalues of a random matrix H by $\lambda_1 \leq \ldots \leq \lambda_N$ and their rescaled versions (\ref{T3.20}) by $\widetilde{\lambda}_1 \leq \ldots \leq \widetilde{\lambda}_N$. Furthermore, let $I_N$ be an interval centred at $a$ and satisfying (\ref{T3.25}). Then the empirical spacing distribution for $H$, localised in $I_N$, is given by
\begin{equation}\label{def_sigma}
\sigma_N(H) \coloneqq \frac{1}{|A_N|} \sum_{\lambda_{i+1},\lambda_i \in I_N} \delta_{\widetilde{\lambda}_{i+1}-\widetilde{\lambda}_{i}}.
\end{equation}
Recall from the discussion above that the expected number of spacings considered in $\sigma_N(H)$ is given by $|A_N|-1$. This explains
the pre-factor $1/|A_N|$ in the definition of $\sigma_N(H)$, which is asymptotically the same as $1/(|A_N|-1)$.

\section{Universality of the $k$-point correlation functions for invariant ensembles}
\label{Tsec:4}
In this section we state results on the convergence of $k$-point correlation functions for invariant ensembles, as well as their connection to orthogonal polynomials. 

We recall that we consider invariant ensembles where the joint distribution of the eigenvalues has a density of the form (see (\ref{T2.10}) and (\ref{T3.15}))
	\begin{equation}\label{joint_density}
	   P_{N}^{(\beta)}(\lambda_1, \ldots, \lambda_N)\coloneqq \frac{1}{Z_{N,\beta}}\prod_{i<j}|
	   \lambda_j-\lambda_i|^{\beta} \prod_{k=1}^N w_N^{(\beta)}( \lambda_k), \quad \lambda \in \mathbb R^N 
	\end{equation}
with $w_{N}^{(\beta)}(x)=e^{-NV(x)}$.
In the proof of the main theorem (Theorem \ref{main_theo}) we will use asymptotic results for the  marginal densities of $P_{N}^{(\beta)}$ with respect to $k$ variables. 
The latter are called the $k$-point correlation functions, for which  we will now give a precise definition. 
\begin{definition}\label{defR} \text{}
		\begin{enumerate}
		\item[(i)] 
			For $k \in \mathbb{N}, k\leq N$, $\beta\in \{1,2,4\}$ and $(\lambda_1,\ldots, \lambda_k) \in \mathbb R^k$ 
			we set 
			   $$ 
			      R_{N,k}^{(\beta)}(\lambda_1, \ldots, \lambda_k)\coloneqq\frac{N!}{(N-k)!}   
			      \int_{\mathbb{R}^{N-k}} 
						P_{N}^{(\beta)}(\lambda_1,\ldots, \lambda_N) \, 
						d\lambda_{k+1}\ldots d\lambda_N.
				 $$ 
		\item[(ii)] 
			For $k \in \mathbb{N}, k\leq N$ and $\beta\in \{1,2,4\}$ the rescaled $k$-point correlation functions are given by
		\begin{align*} 
				B_{N,k}^{(\beta)}(\widetilde{\lambda}_1, \ldots, \widetilde{\lambda}_k)& \coloneqq \left( N \psi(a) 
				\right)^{-k} 
				R_{N,k}^{(\beta)} \left(a+\frac{\widetilde{\lambda}_1}{N\psi(a)}, \ldots,a+\frac{\widetilde{\lambda}_k}
				{N\psi(a)} \right)\\
			     &=\left( N \psi(a) \right)^{-k} 
				R_{N,k}^{(\beta)} \left(\lambda_1, \ldots,\lambda_k \right).
\end{align*}
		\end{enumerate}
\end{definition}
We observe that $R^{(k)}_{N,k}(t_1,\ldots,t_k)$ and $B_{N,k}^{(\beta)}(t_1,\ldots,t_k) $ are invariant under  permutations of the 
 indices $\{1, \ldots, k \}$.

We now sketch how the $k$-point correlation functions can be analysed using the method of orthogonal polynomials. We start with the simplest case $\beta=2$. 
Define $K_{N,2}\colon \mathbb R^2 \to \mathbb R$ 
with
	\begin{equation}\label{OP_1}
	   K_{N,2}(x,y) \coloneqq \sum_{j=0}^{N-1} \varphi_j^{(N)}(x)\varphi_j^{(N)}(y),
	\end{equation}
	\begin{equation*} 
	   \varphi_j^{(N)}(x) \coloneqq p_j^{(N)}(x) \sqrt{w^{(2)}_{N}(x)}, 
	\end{equation*}
and $p_j^{(N)}(x)=\gamma_j^{(N)}x^j+\ldots$ with $\gamma_j^{(N)} >0$ denotes the $j$-th normalised orthogonal 
polynomial with respect to the measure $w^{(2)}_{N}(x)dx$ on $\mathbb{R}$, i.e.\  
	$$\int_{\mathbb{R}}p_j^{(N)}(x)p_k^{(N)}(x)w^{(2)}_{N}(x)dx= \delta_{jk}.$$
The convergence of the appropriately rescaled kernel $K_{N,2}$
\begin{equation}\label{conv_sine_kernel}
 \lim_{N \to \infty }\frac{1}{N \psi(a)} K_{N,2}\left(a+\frac{x}{N \psi(a)},a+\frac{y}{N 
\psi(a)}\right)=\frac{\sin(\pi (x-y))} 
{\pi (x-y)} \eqqcolon K_2(x,y)
\end{equation}
has by now been proved in quite some generality (see e.g.\ \cite{Lub} and references therein). 
Usually uniform convergence of (\ref{conv_sine_kernel}) is only shown for $x,y$ in bounded sets. For our purposes it is convenient to extend this result for $x,y$ in the growing set $A_N$.
%
%
%
\begin{theorem}[c.f.\  \cite{DeiftKriecherbauer}, \cite{Diss}] \label{theo_Kn2}
Let $V \colon \mathbb R \to \mathbb R$ be real analytic such that  (\ref{T3.15}) holds and let $V$ be regular in the sense of \cite[(1.12),(1.13)]{DeiftKriecherbauer}. Moreover, we assume $a \in \mathbb R$ with $\psi(a)>0$
($\psi$ being defined as the density of the minimizer of $I_V$, see (\ref{midnight})). 
Let $(c_N)_{N \in \mathbb N}$ be a sequence satisfying $c_N \to \infty , \frac{c_N}{N} \to 0$ as $N \to \infty$. Then we have for $N \to \infty$
\begin{equation}\label{theo_Kn2:eq1}
 \sup_{x,y \in [-c_N,c_N]} \left|\frac{1}{N \psi(a)} K_{N,2}\left(a+\frac{x}{N \psi(a)},a+\frac{y}{N 
\psi(a)}\right)-K_2(x,y)\right|= \mathcal{O} \left( \frac{c_N}{N} \right)
\end{equation}
\begin{equation}\label{theo_Kn2:eq2}
\sup_{x,y \in [-c_N,c_N]} \left| \frac{\partial}{\partial x} \left( \frac{1}{N \psi(a)} K_{N,2}\left(a+\frac{x}{N \psi(a)},a+\frac{y}{N 
\psi(a)}\right)-K_2(x,y) \right)  \right|
 =  \mathcal{O} \left( \frac{c_N}{N} \right).
\end{equation}
\end{theorem}
%
 %
\begin{remark}\text{}
\begin{enumerate} 
\item[(i)]
The estimate in (\ref{theo_Kn2:eq2})  will be needed to treat the cases $\beta=1$ and $\beta=4$.
\item[(ii)] The proof of Theorem \ref{theo_Kn2} is essentially contained in \cite{DeiftKriecherbauer} although not stated explicitly (a formula somewhat close is presented in \cite[(6.18)]{DeiftKriecherbauer}).
In particular, there is no information on the derivatives in (\ref{theo_Kn2:eq2}). Nevertheless the underlying Riemann-Hilbert analysis also provides (\ref{theo_Kn2:eq1}) and (\ref{theo_Kn2:eq2}), where we use an efficient path, which we have taken from \cite{Van}.
A sketch of the required refinements and the extension to unbounded sets can be found in \cite{Diss}. 
\item[(iii)] To unify the notation with the cases $\beta=1$ and $\beta=4$ treated below we set 
\begin{equation}\label{rescale_K2}
\widehat{K}_{N,2}(x,y)\coloneqq \frac{1}{N \psi(a)} K_{N,2}(x,y)
\end{equation}
and hence (\ref{theo_Kn2:eq1}) reads
\begin{equation}\label{theo_Kn2:eq4}
\widehat{K}_{N,2}\left(  a+\frac{x}{N \psi(a)}, a +\frac{y}{N \psi(a)}  \right) = 	
	K_{2} (x,y) + \mathcal{O} \left( \kappa_N \right)
\end{equation}
with $\kappa_N=\frac{c_N}{N} \to 0$ as $N \to \infty$. The error term is uniform for  $x,y \in [-c_N,c_N]$.
\end{enumerate}
\end{remark}
Theorem \ref{theo_Kn2} can be used to derive some results about the rescaled correlation function, where one uses a well known determinantal formula expressing $B_{N,k}^{(2)}$ in terms of $K_{N,2}$ (see e.g.\  \cite{deift1} and Lemma \ref{lemma_B2} below).  Observe  that in the  considered setting  the term $\mathcal{O}\left( \frac{c_N}{N}\right)$ in the asymptotic behaviour of $K_{N,2}$ (see Theorem \ref{theo_Kn2}) is replaced by $\mathcal{O}(|I_N|)$ in statement (iii) of Lemma \ref{lemma_B2}.
%
	\begin{lemma}\label{lemma_B2}
		Let the assumptions of Theorem \ref{theo_Kn2} be satisfied. 		
		Furthermore,  let  
       $a,\psi, I_N$, $A_N$ be defined as in section \ref{Tsec:3}. 
       Then the following holds
			\begin{enumerate}
				\item[(i)] For $t_1, \ldots, t_k \in \mathbb{R}$  we have 
					$$B_{N,k}^{(2)}(t_1, \ldots, t_k)=(N \psi(a))^{-k} \det\left(
					K_{N,2}\left(a+\frac{t_i}{N\psi(a)},a+\frac{t_j}{N\psi(a)}\right)
					 \right)_{1 \leq i,j\leq k},$$
					 where $K_{N,2}$ is given in (\ref{OP_1}). 
				\item[(ii)]
				 	For $N$ sufficiently large we have for all $k \leq N$
					$$|B_{N,k}^{(2)}(t_1, \ldots, t_k)|\leq 2^k \quad \text{ for } t_1, \ldots, t_k \in 
					A_N.$$ 
				\item[(iii)] For  $t_1, \ldots, t_k \in A_N$   we have
				 	\begin{equation}
				 	\label{Convercence_Bk}	
				 	B_{N,k}^{(2)}(t_1, \ldots, t_k)=W_k^{(2)}(t_1,\ldots,t_k)+k! \cdot k\cdot 2^k
					 \mathcal{O}(|I_N|),
					 \end{equation}
					 with 
					 \begin{equation}\label{Def_W2}
					 W_k^{(2)}(t_1,\ldots,t_k) \coloneqq \det \left( \frac{\sin(\pi(t_i-t_j))}
					{\pi(t_i-t_j)} \right)_{1 \leq i,j \leq k}
					 \end{equation}
					 and the constant implicit in the error term in (\ref{Convercence_Bk}) is uniform in
					  $k$, $t_1,\ldots, t_k$  and in~$N$.					
				\end{enumerate}
	\end{lemma}	
It should be noted that with Lemma \ref{lemma_B2} we have derived all information on the convergence of the $k$-point correlation functions that is needed to prove the main result Theorem \ref{main_theo}  for $\beta=2$.

We now turn to the cases $\beta=1,4$. For technical reasons we restrict the discussion of the case $\beta=1$ to even values of $N$. Our presentation follows closely the monograph \cite{deift2}.
Similar to statement (i) of Lemma \ref{lemma_B2} the $k$-point correlation function for $\beta=1$ and $\beta=4$ can be represented in terms of functions $S_{N,\beta}$, which are related to $K_{N,2}$.
It is convenient to express the correlation functions in terms of the Pfaffian. We remind the reader that for real skew-symmetric $2m \times 2m$ matrices the determinant is a perfect square. Consequently, the Pfaffian which is defined to be the square-root of the determinant for such matrices can be expressed as a polynomial in the entries. Indeed, 
$$\text{Pf} (A) = \frac{1}{2^m m!} \sum_{\sigma \in S_{2m}} (\text{sgn} \sigma)  a_{\sigma_1 \sigma_2 } a_{\sigma_3 \sigma_4 } \ldots a_{\sigma_{2m-1} \sigma_{2m}},$$
where $S_{2m}$ denotes the permutation group on $\{ 1, \ldots, 2m\}$.
See also \cite{deift2}  for an elementary exposition on the use of Pfaffians in Random Matrix Theory.
According to \cite[(4.128),(4.135)]{deift2} the correlation functions can be expressed  via
			\begin{equation}\label{pfaffian}
			R_{N,k}^{(\beta)}(\lambda_1,\ldots, \lambda_k)
			=\text{Pf} ( K\, J) , \quad \text{with } K\coloneqq (K_{N,\beta}(\lambda_i,\lambda_j))_{i,j=1, \ldots, k}  
			\end{equation}
			and
			\begin{equation*} 
			J \coloneqq \text{diag}(\sigma, \ldots , \sigma) \in \mathbb{R}^{2N \times 2N}, \quad \sigma \coloneqq 
			\begin{pmatrix} 0 & 1 \\ -1 & 0 \end{pmatrix}.
			\end{equation*}
In (\ref{pfaffian}) the terms $K_{N,\beta}(x,y), \beta=1,4$ denote $2\times 2$ matrices with 
\begin{align*}
			K_{N,4}(x,y) \coloneqq 	\begin{pmatrix}
						S_{N,4}(x,y) &\frac{\partial}{\partial y} S_{N,4}(x,y) \\
						\\
						- \int_x^y S_{N,4}(t,y) \, dt & S_{N,4}(y,x)
					\end{pmatrix}
\end{align*}
and 
\begin{align*}
		K_{N,1}(x,y) \coloneqq  	\begin{pmatrix}
					S_{N,1}(x,y) & & \frac{\partial}{\partial y} S_{N,1}(x,y) \\ \\
					- \int_x^y S_{N,1}(t,y) \, dt  -\frac{1}{2}\sgn(x-y) &  & S_{N,1}(y,x)
				\end{pmatrix}.
\end{align*}
The convergence of the (rescaled) matrix kernels $K_{N,\beta}$ is e.g.~considered in \cite{deift2}, but as in the case $\beta=2$ the known results only apply to the convergence on compact sets and need to be refined to uniform convergence on $A_N$ (recall $|A_N| \to \infty$ as $N \to \infty$). 
Before we can state Theorem \ref{theo_Kn1_Kn4} we introduce some more notation (in analogy to (\ref{rescale_K2}) for $\beta=2$). 
For $\beta=1,4$ let $\widehat{K_{N,\beta}}(x,y) \in \mathbb R^{2 \times 2}$  
denote a rescaled version of $K_{N,\beta}(x,y)$
given by
	\begin{equation}\label{K_Hut}
		\widehat{K_{N,\beta}}(x,y)
		\coloneqq \frac{1}{N \psi(a)}	
			\begin{pmatrix} 
				\frac{1}{\sqrt{N \psi(a)}} 	& 0 \\
				& \sqrt{N \psi(a)} 
			\end{pmatrix}
		K_{N,\beta}(x,y) 
			\begin{pmatrix} 
				 \sqrt{N \psi(a)}  	& 0 \\
				0 &  \frac{1}{\sqrt{N \psi(a)}} 
			\end{pmatrix}.		
	\end{equation}
We denote the components of the rescaled matrices $\widehat{K_{N,\beta}}(x,y)$ by 
	\begin{equation*}
	   \begin{pmatrix}
	   \widehat{S_{N,\beta}}(x,y) & \widehat{D_{N,\beta}}(x,y) \\
	    \widehat{I_{N,\beta}}(x,y) & \widehat{S_{N,\beta}}(y,x)
	    \end{pmatrix}
	\coloneqq  \widehat{K_{N,\beta}}(x,y).
	\end{equation*}
%
%
%
\begin{theorem}[\cite{Scherbina}, \cite{deift2}, \cite{Diss}]\label{theo_Kn1_Kn4}
Let $V$ be a polynomial of even degree with positive leading coefficient and let $V$ be regular in the sense of \cite[(1.12),(1.13)]{DeiftKriecherbauer}.
Moreover, we assume $a \in \mathbb R$ with $\psi(a)>0$
($\psi$ is defined as in Theorem \ref{theo_Kn2}) and let $K_2$ be given in (\ref{conv_sine_kernel}).  
Let $(c_N)_{N \in \mathbb N}$ be a sequence satisfying $c_N \to \infty , \frac{c_N}{\sqrt{N}} \to 0$ as $N \to \infty$. 
Then we have 
\begin{enumerate}
\item[(i)] For $\beta=1$ and $N$ even 
				\begin{align*}
					&\sup_{x,y \in [-c_N,c_N]}
					\left| \widehat{S_{N,1}}\left(a+\frac{x}{N \psi(a)}, a+\frac{y}{N \psi(a)}\right)
					 - K_{2}(x,y)
				 \right| =\mathcal{O}\left(\frac{1}{\sqrt{N}}\right)	
					\\
					&\sup_{x,y \in [-c_N,c_N]}
					\left| \widehat{D_{N,1}}\left(a+\frac{x}{N \psi(a)}, a+\frac{y}{N \psi(a)}\right)
					- \frac{\partial }{\partial x}K_{2}(x,y)
					\right| =\mathcal{O}\left(\frac{1}{\sqrt{N}}\right)			 
					\\
					&\sup_{x,y \in [-c_N,c_N]}
					\left|
					\widehat{I_{N,1}}\left(a+\frac{x}{N \psi(a)}, a+\frac{y}{N \psi(a)}\right)
					- \int_0^{x-y} K_{2}(t,0)dt -\frac{1}{2}\sgn(x-y)
					\right| = \mathcal{O}\left(\frac{c_N}{\sqrt{N}}\right)
				\end{align*}
\item[(ii)] For $\beta=4$ and $N$ even
				\begin{align*}
				&\sup_{x,y \in [-c_N,c_N]}
					\left|
					\widehat{S_{N/2,4}}\left(a+\frac{x}{N \psi(a)}, a+\frac{y}{N \psi(a)}\right)
					-K_{2}(2(x-y))
					\right| = 
					\mathcal{O}\left(\frac{1}{\sqrt{N}}\right)			 
					\\
					&\sup_{x,y \in [-c_N,c_N]}
					\left|
					\widehat{D_{N/2,4}}\left(a+\frac{x}{N \psi(a)}, a+\frac{y}{N \psi(a)}\right)
					- \frac{\partial }{\partial x}K_{2}(2(x-y))
					\right| =
					\mathcal{O}\left(\frac{1}{\sqrt{N}}\right)		\\
					&\sup_{x,y \in [-c_N,c_N]}
					\left|
					\widehat{I_{N/2,4}}\left(a+\frac{x}{N \psi(a)}, a+\frac{y}{N \psi(a)}\right)
					-\int_0^{x-y} K_{2}(2t)dt  
					\right| = \mathcal{O}\left(\frac{c_N}{\sqrt{N}}\right)
				\end{align*}
				\end{enumerate}
\end{theorem}
\begin{remark}
The proof of Theorem \ref{theo_Kn1_Kn4} can be derived from \cite{Scherbina}, \cite{deift2} and Theorem \ref{theo_Kn2} as follows (details will be given in a later publication):
We use the notation 
$\widehat{x}=a+\frac{x}{N \psi(a)}, \quad \widehat{y}=a+\frac{y}{N \psi(a)}$
and set  
\begin{equation*}
\Delta_{N,\beta} (\widehat{x},\widehat{y})    \coloneqq 
\frac{1}{N \psi(a)} \left( S_{N,\beta}(\widehat{x},\widehat{y})   - K_{N,2}(\widehat{x},\widehat{y}) \right)
=
 \widehat{S}_{N,\beta}(\widehat{x},\widehat{y})-\frac{1}{N \psi(a)}K_{N,2}(\widehat{x},\widehat{y}).
\end{equation*}
As $V$ is a polynomial, we can apply Widom's formalism \cite{WIDOM} 
to derive a representation of $\Delta_{N,\beta} $ in terms of orthogonal polynomials.
Together with the estimates contained in \cite{Scherbina} and \cite[section 6.3.1]{deift2} (generalised to the case of varying weights) we obtain
\begin{align*}
					\sup_{x,y \in [-c_N,c_N]}
					\left| \Delta_{N,\beta} (\hat{x},\hat{y})   \right|
					& = \mathcal{O} \left(   N^{-\frac{1}{2}}\right) \\\
					\sup_{x,y \in [-c_N,c_N]}
					\left|  \frac{1}{N \psi(a)} \frac{\partial}{\partial \hat{y} }\Delta_{N,\beta} (\hat{x},\hat{y})   					\right|
					&= \mathcal{O} \left(   N^{-\frac{1}{2}}\right).
\end{align*}
The claim of Theorem \ref{theo_Kn1_Kn4} then  follows from Theorem \ref{theo_Kn2}
 and from the assumption $\frac{c_N}{\sqrt{N}} \to 0$ for $N \to \infty$, which implies  $\frac{c_N}{N} = \mathcal{O}\left( \frac{1}{\sqrt{N}} \right)$.
\end{remark} 
Finally, we introduce some more notation (recall that $K_2$ was introduced in (\ref{conv_sine_kernel})): 
\begin{align*}
S_1(x,y)&\coloneqq K_2(x,y), \quad D_1(x,y) \coloneqq \frac{\partial }{\partial x}K_{2}(x,y), \\
\quad I_1(x,y) &\coloneqq \int_0^{x-y} K_{2}(t,0)dt -\frac{1}{2}\ \sgn(x-y) 
\end{align*}
\begin{equation*}
S_4(x,y) \coloneqq K_2(2x,2y), \quad D_4(x,y) \coloneqq \frac{\partial }{\partial x}K_{2}(2x,2y),\quad I_4(x,y) \coloneqq \int_0^{x-y} K_{2}(2t,0)dt 
\end{equation*}
and 
\begin{equation*}  
	K_{\beta}(x,y)\coloneqq 
	    \begin{pmatrix}
	       S_{\beta}(x,y) & D_{\beta}(x,y) \\
	       I_{\beta}(x,y) & S_{\beta}(y,x) 
	    \end{pmatrix}.
	\end{equation*}
\begin{remark}\label{remark_theo2}\text{}
\begin{enumerate}
\item[(i)]
With this notation the result of Theorem \ref{theo_Kn1_Kn4} reads:
There exists a sequence $\kappa_N$ such that 	$\kappa_N \to 0$ for $ N \to \infty$  
	and 
\begin{equation}\label{unif_Kn}
\widehat{K_{N,\beta}} \left(a+\frac{x}{N \psi(a)}, a +\frac{y}{N \psi(a)}  \right) = 	
	K_{\beta} (x,y) + \mathcal{O}(\kappa_N)
\end{equation}
uniformly for $x,y \in A_N$. 
\item[(ii)]
Theorems \ref{theo_Kn2} and \ref{theo_Kn1_Kn4} have been stated for invariant matrix ensembles satisfying  (\ref{T2.10}) and (\ref{T3.15}) and  do not cover all 10 symmetry classes (c.f.~section \ref{Tsec:2}). 
However, the statements of Theorem \ref{theo_Kn2} and Theorem \ref{theo_Kn1_Kn4} hold  mutatis mutandis for all invariant ensembles for which universality has been proved using a Riemann Hilbert analysis in the analytic setting (see e.g.~\cite{KuV1},  \cite{DGKV}, \cite{Van} for varying and non-varying Laguerre-type ensembles, \cite{KuV2} for Jacobi-type ensembles and \cite{DKMVZ1} for non-varying Hermite-type ensembles). In this way all symmetry classes are covered. The work of McLaughlin and Miller \cite{McL} shows that one can expect that some finite regularity assumption on $V$ combined e.g.~with the convexity of $V$ should also suffice.
\end{enumerate}
\end{remark}
From (\ref{unif_Kn}) one can deduce the analogue of Lemma \ref{lemma_B2} for $\beta=1,4$ using e.g.~the formulae in \cite{TracyWidom1}. In particular, one can derive the convergence of the rescaled correlation functions $B_{N,k}^{(\beta)}$.  
For $\beta=1,4$ we set (analogue to (\ref{Def_W2}) for $\beta=2$, see also (\ref{pfaffian}))
\begin{equation}\label{def_Wk}
W_k^{(\beta)}(t_1, \ldots, t_k)\coloneqq \text{Pf} (K\, J)  \quad \text{with } K\coloneqq (K_{\beta}(t_i,t_j))_{1\leq i,j\leq k},  \quad t_1,\ldots,t_k \in \mathbb R.
\end{equation}
%
\begin{lemma} \label{lemma_B1_B4}
Suppose  that the assumptions of Theorem \ref{theo_Kn1_Kn4} hold. Then the following holds for $\beta \in \{1,4\}$.
			\begin{enumerate}
				\item [(i)]There exists $C>0$  such that for all $1 \leq k\leq N, \,  t_1,\ldots,t_k \in A_N$ 
					we have
					\begin{equation}
					\label{Konv_B_k}					
					B_{N,k}^{(\beta)}(t_1,\ldots,t_k)=W_{k}^{(\beta)}(t_1,\ldots,t_k)+k!\cdot 
					C^k\,  \mathcal{O}(\kappa_N).
					\end{equation}
					The constant implicit in the $\mathcal{O}$-term is uniform in $k,N$ and in  $t_1, \ldots, t_k$ and $\kappa_N \to 0$ as $N \to \infty$ as in Remark \ref{remark_theo2} (i). 
				\item[(ii)]
				 	The function $W_k^{(\beta)}$ is a symmetric function on ${\mathbb R}^k$ for all $k \in \mathbb N$.
				\item[(iii)] For $k \in \mathbb N$, $t_1,\ldots,t_k$ and $ c \in  \mathbb R$: \quad 
				 	$
				 	W_k^{(\beta)}(t_1+c,\ldots,t_k+c)=W_k^{(\beta)}(t_1,\ldots,t_k).
				 	$
				 	\vspace{4pt}
				\item[iv)] 
				    There exists a constant $C>1$ such that  for all $1 \leq k \leq N$ we have 
				    \begin{align*}
				    \left| B_{N,k}^{(\beta)}(t_1, \ldots, t_k) \right|&\leq C^k \, \, k^{\frac{k}{2}} \quad \text{ for }t_1,
				    \ldots,t_k  \in A_N\\
				    \left| W_{k}^{(\beta)}(t_1, \ldots, t_k) \right|&\leq C^k \, \, k^{\frac{k}{2}}\quad \text{ for }t_1,\ldots,t_k  
				    \in \mathbb R.
				    \end{align*}
		\item[v)] For all $t \in \mathbb R$: \quad 
		$
		W_1^{(\beta)}(t)=1.
		$
			\end{enumerate}
	\end{lemma}
	\begin{remark}
	We note that for the results presented in section \ref{subsec:convergence} it is not necessary to keep track of  the $k$-dependence of the error in (\ref{Konv_B_k}). However, this estimate  is needed in the proof of Theorem \ref{main_theo}.  
	\end{remark}
%

\section{The expected empirical spacing distribution and gap probabilities}
\label{sec:PointwiseConvergence}

The basic result that we want to explain in this section is the convergence of the expected spacing distribution, i.e. 
\begin{equation} \label{sec5:eq1}
\lim_{N \to \infty} \mathbb E_{N,\beta} \left(\int_0^s d \sigma_N(H) \right)
=\int_0^s d\mu_{\beta}
\end{equation}
for some probability measures $\mu_{\beta}$. The limiting spacing distributions $\mu_{\beta}$ depend on $\beta$, but are universal otherwise (see Remark \ref{univ_mu_beta} at the end of this section). 
In our exposition we restrict ourselves to prove the convergence of $\mathbb E_{N,\beta} \left(\int_0^s d \sigma_N(H) \right)$ for $N \to \infty$. This is the content of section \ref{subsec:convergence}. It is not entirely obvious to show that the limit actually defines a probability measure. 
One way to prove this is to make a connection between $\mathbb E_{N,\beta} \left(\int_0^s d \sigma_N(H) \right)$ and the gap probabilities  and to use that the latter can be expressed in terms of Painlev\'{e} V transcendents. We will discuss this connection in section \ref{subsec:gap}.  
%
%
\subsection{Convergence of the expected empirical spacing distribution}\label{subsec:convergence}
In this section we will show the existence of
\begin{align*}
\lim_{N \to \infty} \mathbb E_{N,\beta} \left(  \int_0^s d\sigma_N(H)\right)
\end{align*}
and derive a representation for this limit.
As $ \int_0^s d\sigma_N(H)$ is a function of the ordered eigenvalues of $H$ the expectation is obtained by integration over the Weyl chamber with respect to $B_{N,N}^{(\beta)}(t)dt$ (see Definition \ref{defR}).

The first step in the proof is the introduction of   related counting measures $\gamma_N(k,H)$ for $k \geq 2$.
Recall that the eigenvalues of the random matrix $H$ are denoted by $\lambda_1 \leq  \ldots \leq \lambda_N$ and 
their rescaled versions by $\widetilde{\lambda}_1 \leq  \ldots \leq \widetilde{\lambda}_N$  (see (\ref{T3.20})). We define
\begin{equation} \label{def_gammaN}
	\gamma_N(k,H)\coloneqq  \frac{1}{|A_N|}
	\sum_{\substack{i_1 < \ldots < i_k,\\ \lambda_{i_1},\lambda_{i_k}\in 
	I_N}}\delta_{(\widetilde{\lambda}_{i_{k}}-
	\widetilde{\lambda}_{i_1})}, \quad k\geq 2.
\end{equation}
Observe that the normalizing factor $\frac{1}{|A_N|}$ corresponds to the fact that we expect $|A_N|=N \psi(a) |I_N|$ eigenvalues $\lambda_i\in I_N$ (see discussion below (\ref{def_sigma})).
The measures $\gamma_N(k,H)$ are related to $\sigma_N$ (see Lemma \ref{lemma_combinatoric} below) and the 
 main advantage of $\int_0^s d\gamma_N(k,H)$ over $\int_0^s d\sigma_N(H)$ is that  it is a symmetric function of the eigenvalues of $H$, if we replace $\widetilde{\lambda_{i_{k}}}-
	\widetilde{\lambda_{i_1}}$ in (\ref{def_gammaN}) by  $\max_{1 \leq j \leq k} \widetilde{\lambda_{i_j}} - \min_{1 \leq j \leq k} \widetilde{\lambda_{i_j}}$. 
This allows  us to calculate the expectation of $\int_0^s d\gamma_N(k,H)$ by integration over $\mathbb R^N$ (instead of $\mathcal{W}_N$) with respect to $\frac{1}{N!}B_{N,N}^{(\beta)}(t)dt$ (see (\ref{joint_density})) .
Thus we can exploit the invariance of the $k$-point correlation functions under permutations of the arguments together with their uniform convergence given in Lemma \ref{lemma_B2} resp.~in Lemma \ref{lemma_B1_B4}. 

By combinatorial arguments (see e.g.~Corollary 2.4.11, Lemma 2.4.9 and Lemma 2.4.12 in \cite{katzsarnak}) one can show the following connection between $\sigma_N(H)$ and $\gamma_N(k,H)$.
\begin{lemma}[c.f.~chap.~2 in \cite{katzsarnak} ]\label{lemma_combinatoric}
\begin{enumerate}
\item[(i)]For $N \in \mathbb N$ we have
\begin{equation}\label{combinatorics}
\int_{0}^s  \, d\sigma_N(H)= \sum_{k=2}^N (-1)^k \int_0^s
\, d \gamma_N(k,H).
\end{equation}
\item[(ii)] For $N \in \mathbb N$ and $m \leq N$ we have
\begin{align*}
\int_0^s d\sigma_N(H) &\geq \sum_{k=2}^m (-1)^k \int_0^s d\gamma_N(k,H) \quad \text{for }m \text{ odd} \\
\int_0^s d\sigma_N(H) &\leq \sum_{k=2}^m (-1)^k \int_0^s d\gamma_N(k,H) \quad \text{for }m \text{ even}.
\end{align*}
\end{enumerate}
\end{lemma}
We can use Lemma \ref{lemma_combinatoric} to prove the following theorem, which states the point wise convergence of the empirical spacing distribution.
\begin{theorem}[c.f.~\cite{deift1} for $\beta=2$]\label{pointwise}
Suppose that the assumptions of Theorem \ref{theo_Kn2} ($\beta=2$) resp.~of Theorem \ref{theo_Kn1_Kn4} ($\beta=1,4$) are satisfied.  Then we have for $\beta=1,2,4$,  $s \in \mathbb R$ and $W_k^{(\beta)}$ as in (\ref{def_Wk}) ($\beta=1,4$) resp.~in  (\ref{Def_W2}) ($\beta=2$)
	\begin{equation} \label{Darst_mu_beta}
		\lim_{N \to \infty} \mathbb E_{N,\beta} \left(  \int_0^s d\sigma_N(H)\right) = \sum_{k\geq 			2} (-1)^k \int_{0 \leq z_2 \leq \ldots \leq z_k \leq s}  W_k^{(\beta)}(0, z_2,		
		\ldots,z_k) dz_2 \ldots dz_k.
	\end{equation}
	In particular, we claim that the series on the right hand side of the equation converges. 
\end{theorem}
%
%
\begin{proof} 
The proof is in the spirit of \cite[chap.~5]{katzsarnak}. 
Taking expectations in (\ref{combinatorics}) leads to
	\begin{equation}  \label{combi_expect}
		\mathbb E_{N,\beta} \left( \int_0^s d\sigma_N(H)  \right)= \sum_{k= 2}^N (-1)^k  	
		\mathbb E_{N,\beta} \left(\int_0^s d\gamma_N(k,H) \right).
	\end{equation}
	We start with the calculation of the expectation on the right hand side of (\ref{combi_expect}).
Observe that we can rewrite
	\begin{equation*}
	\int_0^s d\gamma_N(k,H)= \frac{1}{|A_N|}\sum_{T \subset \{1,\ldots, N\}, |T|=k} \chi (\widetilde{\lambda}_T) 
	\end{equation*}
with
	\begin{equation*}
 	\widetilde{\lambda}_T \coloneqq (\widetilde{\lambda}_{i_1},  \ldots,  \widetilde{\lambda}_{i_k}) \quad \text{ for } T=\{i_1, \ldots, i_k\} \text{ with } 1 \leq i_1 < \ldots < i_k \leq N
	\end{equation*} 
and 
	\begin{equation*}
	\chi(t_t,\ldots,t_k)\coloneqq \chi_{(0,s)} \left(\max_{i=1,\ldots,k} t_i 	-\min_{i=1,\ldots,k} t_i \right) \prod_{i=1}^k \chi_{A_N} (t_i),
	\end{equation*}
where $\chi_{(0,s)}$ resp.~$\chi_{A_N}$ denote the characteristic functions on $(0,s)$ resp.~on $A_N$.

Using the symmetry of the joint density of the eigenvalues with respect to permutations of the variables (see (\ref{joint_density}) and Definition \ref{defR})  we conclude
	\begin{align}
		\mathbb E_{N,\beta} \left(\int_0^s d\gamma_N(k,H) \right)
		=& \frac{1}{N!} \int_{\mathbb R^N} \left( \frac{1}{|A_N|} \sum_{T 
		\subset 
		\{1,\ldots, N\}, |T|=k} \chi (t_T) \right)
		B_{N,N}^{(\beta)}(t) dt 
		\nonumber
		\\
		=& \frac{1}{|A_N|} \frac{1}{N!} \binom{N}{k} \int_{\mathbb R^N}    
		\chi (t_1,\ldots,t_k) 			 
		B_{N,N}^{(\beta)}(t) dt
		\nonumber
		\\
		=&  \frac{1}{|A_N|}\int_{\mathcal{W}_k \cap A_N^k}  \chi (t_1,
		\ldots,t_k)B_{N,k}^{(\beta)}(t_1,\ldots,t_k) dt_1 \ldots dt_k.
		\label{Darst_E_gamma}
	\end{align} 
It is straightforward to prove the following bound
	\begin{equation*} 
		\frac{1}{|A_N|}\int_{\mathcal{W}_k \cap A_N^k}  \chi (t_1,			\ldots,t_k) 
		dt_1 \ldots dt_k \leq \frac{s^{k-1}}{(k-1)!}
	\end{equation*}
which, together with the uniform convergence of  Lemma \ref{lemma_B2} (iii) and Lemma \ref{lemma_B1_B4} (i), leads to 
	\begin{align*}
		 &\frac{1}{|A_N|}\int_{\mathcal{W}_k \cap A_N^k}  \chi 				(t_1,
		 \ldots,t_k)B_{N,k}^{(\beta)}(t_1,\ldots,t_k) dt_1 \ldots 		dt_k 
		 \\
		 =&  \frac{1}{|A_N|}\int_{\mathcal{W}_k \cap A_N^k}  \chi 			(t_1,
		 \ldots,t_k)W_k^{(\beta)}(t_1,\ldots,t_k) dt_1 \ldots 			dt_k +  
		 \mathcal{O}_{s,k} (\kappa_N),
	\end{align*}
where the constant implicit in the $\mathcal{O}$-notation may depend on $s$ and $k$ as indicated by the subscripts.
Using the translation invariance of $W_k^{(\beta)}$ (see (\ref{Def_W2}) for $\beta=2$ and Lemma \ref{lemma_B1_B4} (iii) for $\beta=1,4$) together with the change of variables $z_1=t_1, z_i=t_i-t_1, i=2,\ldots, k$ and the definition of $\chi$ we have
	\begin{align*}
		&\frac{1}{|A_N|}\int_{\mathcal{W}_k \cap A_N^k}  \chi (t_1,
		\ldots,t_k)W_k^{(\beta)}(t_1,\ldots,t_k) dt_1 \ldots 			dt_k 
		\\
		=&  \int_{0 \leq z_2 \leq \ldots \leq z_k \leq s}  W_k^{(\beta)}(0, z_2,
		\ldots,z_k) dz_2 \ldots 			dz_k 
		\\
		-&\frac{1}{|A_N|}\int_{A_N} \left(
		\int_{0 \leq z_2 \leq \ldots \leq z_k \leq s}    W_k^{(\beta)}(0, z_2,\ldots,z_k) 
		\left( 1-\prod_{j=2}^k \chi_{A_N} (z_1+z_j)\right) dz_2 \ldots 			dz_k 
		\right)dz_1
		\\ 
		=& \int_{0 \leq z_2 \leq \ldots \leq z_k \leq s}  W_k^{(\beta)}(0, z_2,\ldots,z_k) 
		dz_2 \ldots dz_k + \mathcal{O}_{s,k} \left( \frac{1}{|A_N|}\right).
	\end{align*}
Hence we obtain  
\begin{equation}\label{Konv_gamma}
\lim_{N \to \infty}  \mathbb E_{N,\beta} \left(\int_0^s d\gamma_N(k,H) \right) = \int_{0 \leq z_2 \leq \ldots \leq z_k \leq s}  W_k^{(\beta)}(0, z_2,\ldots,z_k) 
		dz_2 \ldots dz_k.
\end{equation}
For later reference we observe that by 
the upper bounds on $W_k^{(\beta)}$ provided in Lemma \ref{lemma_B2} ($\beta=2$) and in Lemma \ref{lemma_B1_B4} ($\beta=1,4$) we have
  \begin{equation} \label{schranke_W}
 \lim_{N \to \infty}  \mathbb E_{N,\beta} \left(\int_0^s d\gamma_N(k,H) \right) \leq 
 C^k s^{k-1} \left( \frac{1}{\sqrt{k-1}} \right)^{k-1}.
  \end{equation} 
It remains to show that in (\ref{combi_expect}) the limit $N \to \infty$ may be interchanged with the infinite summation over $k$.
 Taking expectations in  Lemma \ref{lemma_combinatoric} (ii) together with the convergence in (\ref{Konv_gamma}) implies for $m$ odd
  \begin{align}
\sum_{k=2}^m (-1)^k  \lim_{N \to \infty} \mathbb E_{N,\beta} \left(\int_0^s d\gamma_N(k,H) \right) \leq \liminf_{N \to \infty} \,   \mathbb E_{N,\beta} \left(  \int_0^s d\sigma_N(H)\right) 
  \label{theorem3:eq1}
  \\
  \limsup_{N \to \infty} \,   \mathbb E_{N,\beta} \left(  \int_0^s d\sigma_N(H)\right)
  \leq \sum_{k=2}^{m+1} (-1)^k  \lim_{N \to \infty}  \mathbb E_{N,\beta} \left(\int_0^s d\gamma_N(k,H) \right). 
  \label{theorem3:eq2}
  \end{align}
  Inequality (\ref{schranke_W}) ensures the convergence of the series in (\ref{theorem3:eq1}) and (\ref{theorem3:eq2}) if we take $ m \to \infty$.
  Sending   $ m \to \infty$    in (\ref{theorem3:eq1}) and (\ref{theorem3:eq2}) implies that the limit $\displaystyle    \mathbb E_{N,\beta} \left(  \int_0^s d\sigma_N(H)\right)$ exists for $N \to \infty$.  We obtain 
  \begin{equation*}
    \lim_{N\to \infty} \mathbb E_{N,\beta} \left(  \int_0^s d\sigma_N(H)\right)
    = \sum_{k=2}^{\infty} (-1)^k  \lim_{N \to \infty}  \mathbb E_{N,\beta} \left(\int_0^s d\gamma_N(k,H) \right)
  \end{equation*}
  which, together with (\ref{Konv_gamma}), completes the proof.  
\end{proof}
\begin{remark}\label{univ_mu_beta}
In the above theorem we have obtained a representation for the limit of the expected spacing distribution in terms of $W_k^{(\beta)}$, which is hence universal in the sense that the limit does neither depend on $V$ nor on the details of the localisations, i.e.~on  the point $a$ or on  the interval $I_N$ as long as the assumptions of Theorem \ref{theo_Kn2} resp.~Theorem \ref{theo_Kn1_Kn4} are satisfied. 
\end{remark}

However, formula  (\ref{Darst_mu_beta}) is somewhat complicated. In the next section we show how it is related to the so-called gap probabilities that have an explicit representation in terms of particular Painlev\'{e} V functions.
 %
%
%
\subsection{Spacing distributions and gap probabilities}
\label{subsec:gap}
A gap probability is  the probability of having no eigenvalues in a given interval. 
Observe that for finite $N$ and $\beta=2$ we have (see e.g.~\cite[p.~108]{deift1})
	\begin{align*}
	&  \mathbb P_{N,2}(\{ \widetilde{\lambda}_1,\ldots, \widetilde{\lambda}_N\} \cap (0,s) = \emptyset) 
	= \frac{1}{N!}\int_{(\mathbb R \setminus (0,s))^N} B_{N,N}^{(2)} (t_1,\ldots,t_N) \, dt_1\ldots 
	dt_N 
	\\ 
	=& \sum_{k=0}^{N} \frac{(-1)^k}{k!} \int_0^s \ldots \int_0^s \det \left( K_{N,2} 	
	(t_i,t_j)\right)_{1\leq i,j \leq k} dt_1 \ldots  dt_k 
	\\
	= &\det (1- K_{N,2}|_{L^2(0,s)}). 
	\end{align*}
	Here $K_{N,2}|_{L^2(0,s)}$ denotes the integral operator  on $L^2(0,s)$  with kernel $K_{N,2}$ and the last equality is just a standard expansion for the corresponding Fredholm determinant. 
Recall that  $K_{N,2} \to K_2$ for $N \to \infty$ (Theorem \ref{theo_Kn2}).
Furthermore, one can also show that  the corresponding Fredholm determinants converge (see \cite{deift2} and also \cite{Widom_matrixkernel}). 
This motivates that the large $N$-limit
	\begin{equation*}
	G_2(s)\coloneqq \det (1- K_2|_{L^2(0,s)})
	\end{equation*}
 is called the gap probability (for $\beta=2$).
For $\beta=1$ and $4$ the gap probabilities $G_{\beta}$ are defined as square roots of determinants of operators on $L^2(0,s) \times L^2(0,s)$ (see e.g.~\cite[corollary 6.12]{deift2}).

By the standard expansion of the Fredholm determinant we have 
	\begin{equation}\label{gapprob1}
	G_2(s)=\sum_{k=0}^{\infty} \frac{(-1)^k}{k!} \int_0^s \ldots \int_0^s   W_k^{(2)}(t_1,	
	\ldots,t_k)  dt_1 \ldots  dt_k.  
	\end{equation}
Using more involved arguments 
the analogue of equation (\ref{gapprob1}) (with $W_k^{(2)}$ replaced by $W_k^{(\beta)}$ and $G_2$ replaced by $G_{\beta}$) can also  be shown for $\beta=1$
 and $\beta=4$ (see e.g.~\cite[sec.~7.1]{Diss}).
The following theorem  relates the derivatives of the gap probabilities to the limiting spacing distributions for all $\beta \in \{1,2,4 \}$  (see e.g.~\cite[p.~126]{deift1} for $\beta=2$).
%
%
\begin{lemma}\label{theo_gap}
For $\beta=1,2,4$ we have 
\begin{equation*}
-G_{\beta}'(s) =1- \lim_{N \to \infty} \mathbb E_{N,\beta} \left(  \int_0^s \,d\sigma_N (H) \right)
\end{equation*}
\end{lemma}
\begin{proof}
We introduce the function 
	\begin{align}
	\widetilde{G}_{\beta}(\varepsilon,s)  \coloneqq & \sum_{k=0}^{\infty} \frac{(-1)^k}{k!} 	
	\int_{\varepsilon}^s \ldots \int_{\varepsilon}^s   W_k^{(\beta)}(t_1,\ldots,t_k)  dt_1 
	\ldots  dt_k
	\nonumber
	 \\
	=& 1-s+ \varepsilon + \sum_{k = 2}^{\infty} (-1)^k \int_{\varepsilon}^s \left(     	
	\int_{t_1\leq t_2\leq \ldots \leq t_k \leq s}   W_k^{(\beta)}(t_1,\ldots,t_k)  dt_2 	
	\ldots  dt_k   \right) dt_1. \label{gapeq:1}
	\end{align}
	Here we have used $W_1^{(\beta)} (t) =1$ for all $t \in \mathbb R$ (c.f.~Lemma \ref{lemma_B2} and Lemma \ref{lemma_B1_B4}).
Then the translation invariance of $W_k^{(\beta)}$ implies $\widetilde{G}_{\beta}(\varepsilon,s)=\widetilde{G_{\beta}}(0,s-\varepsilon) =G_{\beta}(s-\varepsilon)$ and hence
$\frac{\partial}{\partial \varepsilon}\big|_{\varepsilon=0} \widetilde{G}_{\beta} (\varepsilon,s) =- G_{\beta}'(s)$. Differentiating each term of the series in (\ref{gapeq:1}) (which is absolutely convergent, see (\ref{schranke_W})) we obtain the desired result from Theorem \ref{pointwise}.
\end{proof}
\begin{remark}
As mentioned above there is a remarkable identity that allows to express the gap probabilities $G_{\beta}, \beta \in \{1,2,4\}$ in terms of  Painlev\'{e} V functions. More precisely, let 
 $\sigma$ be the solution  of 
	\begin{equation*}
	(s \sigma'')^2 + 4 (s \sigma'-\sigma) (s \sigma'-\sigma + (\sigma')^2)
	\end{equation*} 
with boundary condition 
	\begin{equation*}
	\sigma(s) \sim -\frac{s}{\pi} - \frac{s^2}{\pi^2}- \frac{s^3}{\pi^3} + \mathcal{O}(s^4) 	\quad \text{for }s \to 0.
	\end{equation*}
Then in the case $\beta=2$ we have (see \cite{jimbo})
	\begin{equation*}
	G_2(s) = \exp\left( \int_0^{\pi s} \frac{\sigma(t)}{t} \, dt\right).
	\end{equation*}
	For $\beta=1$ and $\beta=4$ see \cite{anderson} and  \cite{forresterwignertype} for analogue formulae.

\end{remark}
We recall that Lemma \ref{theo_gap} together with the  Painlev\'{e}  representations for $G_{\beta}$ are useful to verify that $\mu_{\beta}$ as defined through (\ref{sec5:eq1}) is indeed a probability measure (see \cite[chap.~6 and 7]{Diss}).
%
%
\section{Results}\label{sec:result}
In this section we state our new result (Theorem \ref{main_theo}) for the expected
empirical spacing distribution for invariant orthogonal and symplectic ensembles. 
We include a brief discussion of related results that can be found in the literature. 

Except for the point wise convergence for $\beta=2$ presented in Theorem  \ref{pointwise}  (c.f.~\cite{DeiftKriecherbauer}) the empirical spacing distribution has so far only been considered for circular ensembles. In the case $\beta=2 $ the circular unitary ensemble (CUE) is given by the unitary group $U(N)$ with the normalised translation invariant Haar measure. 
The joint distribution of the complex eigenvalues $e^{i \theta_1},\ldots,e^{i \theta_N}$ in the CUE and  in the related orthogonal and symplectic ensembles ($\beta=1,4$) is  given by
	\begin{align}\label{circular_measure}
	d\mathbb P_{N,\beta} (\theta) = \frac{1}{Z_{N,\beta}} \prod_{j<k} \left|  	
	e^{i\theta_k}-e^{i\theta_j}\right|^{\beta} \, d\theta_1 \ldots d\theta_N, \quad 	
	\beta=1,2,4. 
	\end{align}
As the expected spectral density is  constant for these ensembles, the eigenvalues can be normalised to have mean spacing one by the linear rescaling
	\begin{align*}
	\widetilde{\theta_i}\coloneqq \frac{N \theta_i}{2\pi}.
	\end{align*}
Observe that we do not need to localise the spectrum in these cases.

It is a result of  Katz and  Sarnak in \cite[chap.~1 and chap.~2]{katzsarnak} that for circular ensembles with $\beta=2$ 
the expected empirical spacing distribution converges to the same measure $\mu_2$ that we have defined in (\ref{sec5:eq1}). Moreover, they show a stronger version of convergence, i.e.~the vanishing of the expected Kolmogorov distance
	\begin{align}\label{Katz_Sarnak}
	\lim_{N\to \infty} \mathbb E_{N,\beta} \left( \sup_{s \in \mathbb R} \left| \int_0^s 	
	d\sigma_N(H) - \int_0^s d\mu_2  \right|\right) = 0.
	\end{align}
Here $\sigma_N$ is defined as the (normalised) counting measure of the nearest neighbour 
spacings between $\widetilde{\theta}_j\,$'s. The definition of $\sigma_N$ is similar to (\ref{def_sigma}) with the pre-factor altered to $1/(N-1)$ and without the restriction to $I_N$ in the sum.

In \cite{soshnikov}  the convergence in (\ref{Katz_Sarnak}) is sharpened, proving almost sure convergence, and generalised to COE ($\beta=1$), but not to $\beta=4$. Moreover, Soshnikov shows in \cite{soshnikov} for both CUE and COE a central limit theorem for spacings. For example, he proves that the appropriately    normalised  random variables 
	$$ \xi_N (s)= \frac{ \int_0^s d\sigma_N(H) - \mathbb E_{N,\beta}\left( \int_0^s d\sigma_N(H)  	\right)}{ \sqrt{N}}$$
converge  to a Gaussian process $\xi$ with $\mathbb E(\xi(s))=0$ and for which $\mathbb E(\xi(s) \xi(t))$ can be expressed in terms of the $k$-point correlations  of  (\ref{circular_measure}).

Another interesting result \cite[sec.~4.2]{anderson} concerns the  theory of determinantal  point processes. In \cite{anderson}  it is shown that for such point processes with constant intensities generated by a suitable class of kernels (including in particular the sine-kernel $K_2$)
the linear statistics of the empirical spacing distribution converge almost surely to the linear statistics of $\mu_2$ as the number of considered points tends to infinity. This result does not deal with the distribution of the eigenvalues of random matrices for finite $N$.
Nevertheless, it is conceivable that this result might be useful for proving the convergence of the empirical spacing distribution. 


We now turn to our recent results. 
We show in \cite{Diss} that the analogue version of  (\ref{Katz_Sarnak}) is valid for orthogonal and symplectic invariant ensembles  satisfying (\ref{T2.10}) and (\ref{T3.15}). In fact, we reduce the question of convergence of the expected Kolmogorov distance of the empirical spacing distribution  to the convergence of the corresponding kernel functions.
All information that is needed on the convergence
 $K_{N,\beta} \to K_{\beta}$
is summarised in the following  
\begin{assumption}\label{main_assumption}
We consider invariant ensembles with joint distribution of the eigenvalues given by  (\ref{T2.10}) and (\ref{T3.15})  for $\beta \in \{1,2,4\}$. We assume that the limiting spectral density exists and we choose  $a$, $I_N$ and the rescaling of the eigenvalues as in section \ref{sec:3_2}. 
We assume that there exists  a sequence $\kappa_N$ such that $\kappa_N \to 0$ for $ N \to \infty,$ such that for the rescaled (matrix) kernels $\widehat{K}_{N,\beta}$ (see (\ref{rescale_K2}) and (\ref{K_Hut})) 
we have 
\begin{align}\label{assumption}
	\widehat{K_{N,\beta}} \left(a+\frac{x}{N \psi(a)}, a +\frac{y}{N \psi(a)}  \right) = 	
	K_{\beta} (x,y) + \mathcal{O}(\kappa_N)
	\end{align} 
	uniformly for $x,y \in A_N$.
\end{assumption}
%
%
%
%
Our main theorem then reads 
	\begin{theorem}[\cite{Diss}]\label{main_theo}
	Under Assumption \ref{main_assumption} we have 
	\begin{align}\label{main_theo:eq}
	\lim_{N\to \infty} \mathbb E_{N,\beta} \left( \sup_{s \in \mathbb R} \left| \int_0^s 	
	d\sigma_N(H) - \int_0^s d\mu_{\beta}  \right|\right) = 0,
	\end{align}
	where $\mu_{\beta}$ is defined through (\ref{sec5:eq1}).
  	\end{theorem}
In particular, our theorem covers all invariant ensembles for which the convergence of $K_{N,\beta}$ to $K_{\beta}$ 
has been proved using a Riemann-Hilbert approach (see e.g.~Theorem \ref{theo_Kn2} and \ref{theo_Kn1_Kn4}). 
Observe that our formulation of Theorem \ref{main_theo} also includes all ensembles for which (\ref{assumption}) will be established in the future. 

The proof  follows the path devised by Katz and Sarnak in \cite{katzsarnak} for $\beta=2$ and extends their methods in two ways. On the one hand, we have to consider the additional localisation that is needed in our setting.  We can use the same methods as in \cite{katzsarnak} to express the expected  empirical distribution of the spacings in terms of the rescaled $k$-point correlation functions $B_{N,k}^{(\beta)}$ (see (\ref{combi_expect}) and  (\ref{Darst_E_gamma})). On the other hand, we generalise their methods to $\beta=1$ and $\beta=4$. Here  the relation between the matrix-kernel functions $K_{N,\beta}$ and the expected empirical  spacing distribution is   more involved. Moreover, for $\beta=4$ subtle cancellations have to be used to establish convergence.

The proof of the main  theorem comes in three steps: 
The first step is the point wise convergence as shown in Theorem  \ref{pointwise}.   This convergence is well known although it seems that the details have so far only been worked out in the case $\beta=2$ (see e.g.\ \cite{deift2}, \cite{anderson}). In order to obtain the convergence of 
	\begin{align} \label{exp_Kolmog}
	\mathbb E_{N,\beta} \left(  \left| \int_0^s \, d\sigma_N(H) - \int_0^s 	
	\,d\mu_{\beta} \right|\right)
	\end{align}
for any given $s \in \mathbb R$, we bound the variance of $\int_0^s \,d\gamma_N(k,H)$ in the second step. As stated above, this is the most challenging part in  generalising the method of Katz and Sarnak to $\beta=1,4$. Here we found the representation of the $k$-point correlation functions in terms of $K_{N,\beta}$ as provided in \cite{TracyWidom1} useful.  Finally, the desired result is obtained by controlling the $s$-dependence of the bound on (\ref{exp_Kolmog}) together with tail estimates on $\mu_{\beta}$.  
Details of the proof can be found in \cite{Diss}.
%
%
\section{Numerical results}
\label{sec:numerik}
%
In addition to the analytical considerations that led to Theorem \ref{main_theo}, the work \cite{Diss} also contains numerical experiments in MATLAB in order to determine the rate of convergence in  (\ref{main_theo:eq}). 
We summarise some of the findings of \cite{Diss} in the present section.

We conduct our experiments for the three classical Gaussian ensembles GOE, GUE, GSE and for general $\beta$-ensembles with $\beta \in \{7,15.5,20\}$. We also include 
 real, complex and quaternionic Wigner matrices with i.i.d.~entries that are drawn e.g.~from  beta, poisson, exponential, uniform or chi-squared distributions. 
Observe (see section \ref{sec:3_1}) that  in all these cases the limiting spectral density $\psi$ is given by the Wigner semi-circle law. We may adapt the parameters such that the support of $\psi$ is the interval $[-1,1]$. 

A little thought shows that the localisation and rescaling procedure to define $d\sigma_N$ (see (\ref{def_sigma})) will not lead to an optimal and natural rate of convergence. 
Firstly, the rate will depend on the number of eigenvalues, i.e.~on the length of $I_N$. Secondly, the linear rescaling (\ref{T3.20}) is not optimal since the density $\psi$ is approximated on all of $I_N$ by the constant $\psi(a)$. A far better rescaling in this respect (but less suitable for analytical considerations) is the so-called unfolding, that we explain now.  
Let $I \subset [-1,1] = \text{supp} (\psi)$ be an interval. Denote by 
\begin{equation*}
F(t) \coloneqq \frac{2}{\pi} \int_0^t \sqrt{1-s^2} \chi_{[-1,1]} ds
\end{equation*}
the distribution function of the semi-circle law. The rescaling is then given by 
\begin{equation*}
\widetilde{\lambda}_i \coloneqq N F(\lambda_i).
\end{equation*}
Observe that in an average sense
\begin{equation*}
\widetilde{\lambda}_{i+1} -\widetilde{\lambda}_{i}  \approx N F'(\lambda_i) (\lambda_{i+1}- \lambda_i) \approx 
N \psi(\lambda_i) \frac{1}{N\psi(\lambda_i)}=1.
\end{equation*}
The spacing distribution corresponding to the unfolded statistics is given by 
\begin{equation*}
\sigma_N^{(\text{unf})}(H) \coloneqq \frac{1}{|\{i: \lambda_i \in I \} |-1}
\sum_{\lambda_i, \lambda_{i+1}\in I} 
\delta_{\widetilde{\lambda}_{i+1}-\widetilde{\lambda}_i}. 
\end{equation*}
We restrict our attention to intervals $I=[-0.1;0.1]$, $I=[-0.5;0.5]$,   $ I=[-0.75;0.75]$ and  some non-centred intervals such as $I=[0.4;0.6]$

We provide numerical evidence for the claim that   the leading asymptotic of the considered expected Kolmogorov distance is $CN^{-1/2}$, i.e.\
\begin{equation}
\label{report}
E_N \coloneqq \mathbb{E}_{N,\beta} \left(\sup_{s \in \mathbb{R}} \left|\int_{-\infty}^s \, 
  d\sigma_N^{\text{(unf)}}(H) - \int_{-\infty}^s \, d\mu_{\beta}\right|\right)
   \sim C N^{-1/2}
\end{equation}  
for some constant $C$ that depends mildly on the chosen ensemble and on the choice of $I$. 
In Figures \ref{fig:1} to \ref{fig:3}  we have plotted $y\coloneqq -\log E_N$ against $x \coloneqq \log N$. 
We see in all three cases that our numerical approximations to $y$, obtained  by Monte Carlo simulations, cluster impressively close to a straight line
\begin{equation*}
y(x) \sim ax +b, \quad \text{i.e.~} \quad E_N \sim e^{-b} N^{-a}.
\end{equation*}
In all our experiments \cite{Diss} we found $a \in [0.48; 0.53]$ and $b \in [-1,0.5]$, see also Table \ref{tab:1}.
 \begin{figure}
\begin{overpic}[clip=true, trim=0.5cm 0.4cm 0.5cm 0.6cm, width=7cm,]{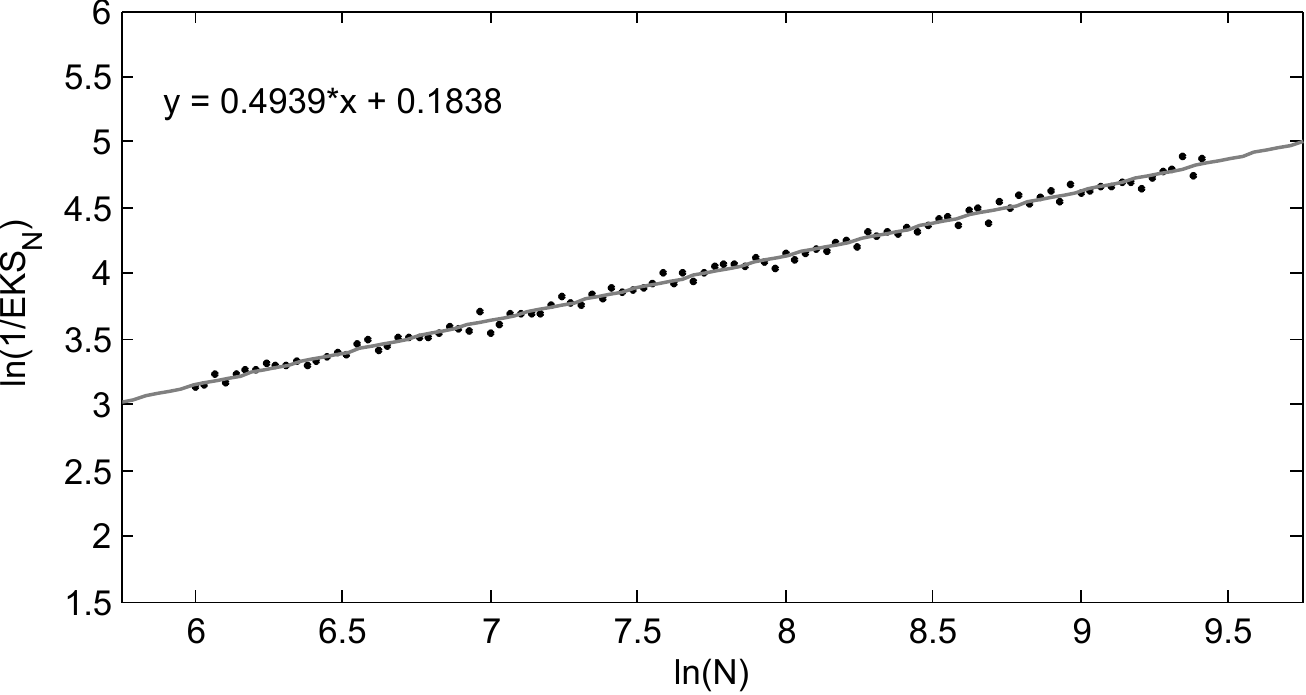}
\put(45,-4){$\log(N)$}
\put(-5,15){\rotatebox{90}{$-\log(E_N)$}}
\end{overpic}
\caption{Data set and best linear fit for $\beta$-ensemble 
with $\beta=4$ and $I=[-0.5,0.5]$}
\label{fig:1}       
\end{figure}

\begin{figure}
\begin{overpic}[clip=true, trim=0.5cm 0.4cm 0.5cm 0.6cm, width=7cm]{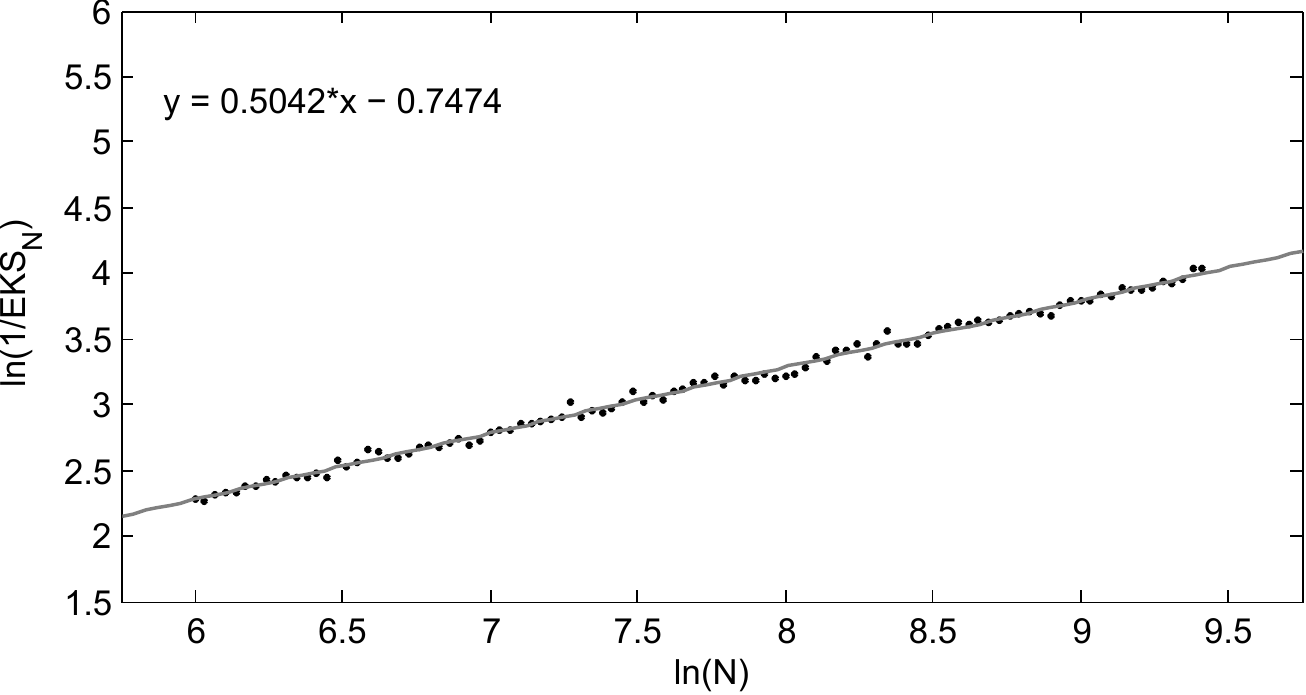}
\put(45,-4){$\log(N)$}
\put(-5,15){\rotatebox{90}{$-\log(E_N)$}}
\end{overpic}
\caption{Data set and best linear fit for $\beta$-ensemble 
with
$\beta=20$ and $I=[0.4,0.6]$.}
\label{fig:2}       
\end{figure}
%

\begin{figure}
\begin{overpic}[clip=true, trim=0.5cm 0.4cm 0.5cm 0.6cm, width=7cm]{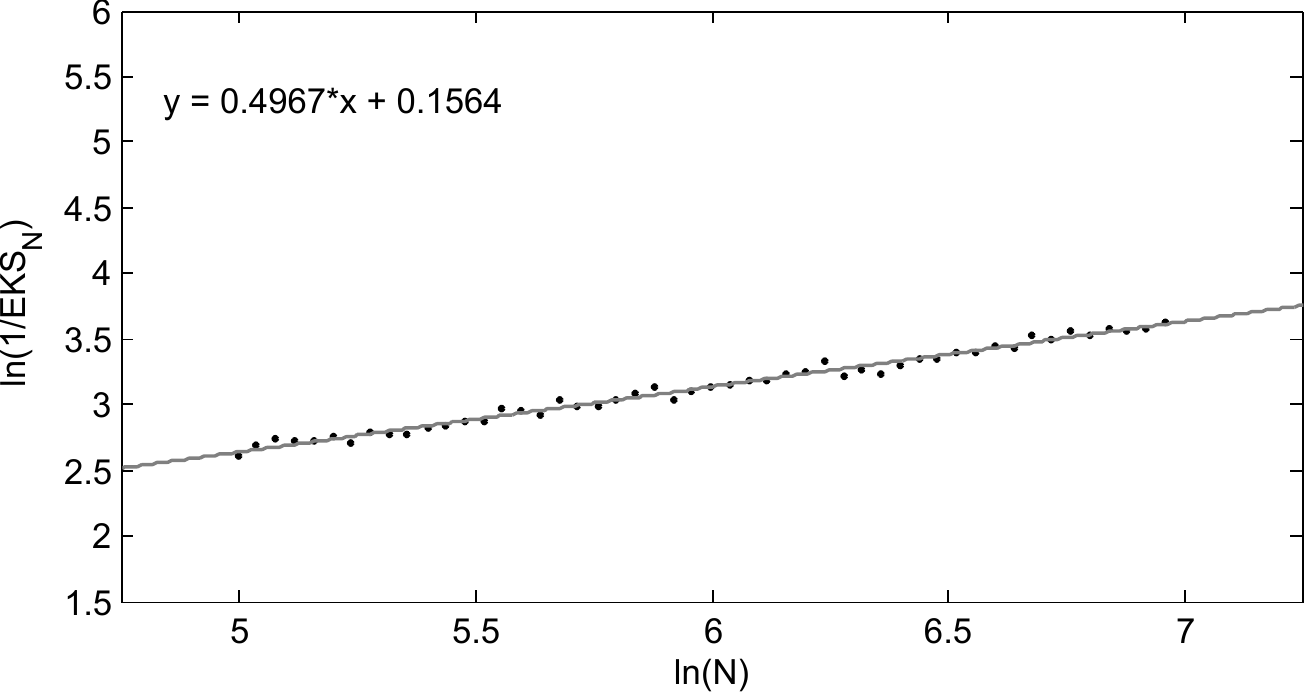}
\put(45,-4){$\log(N)$}
\put(-5,15){\rotatebox{90}{$-\log(E_N)$}}
\end{overpic}
\caption{Data set and best linear fit for real Wigner matrices with unfolded statistics  and beta (2,5)-distributed entries    and $I=[-0.5,0.5]$.}
\label{fig:3}    
\end{figure}

\begin{table}[h]
\centering
\caption{Best fit straight lines for $\beta$-ensembles with unfolded statistics}
\label{tab:1}       
\begin{tabular}{p{1.5cm}p{2.8cm}p{3.6cm}}
\hline\noalign{\smallskip}
ensemble & interval $I$ & best linear fit   \\
\noalign{\smallskip}\hline\noalign{\smallskip}
$\beta=1$ & $I=[ 0.7; 0.9]$  & $y=0.5077x-0.9758$  \\
$\beta=2$ &$I=[-0.1 ;0.1 ]$ & $y=0.4958x-0.6325$\\
$\beta=4$ & $I=[-0.75 ;0.75 ]$ & $y=0.4965x+0.3356$\\
$\beta=7$ & $I=[ -0.1; 0.1]$  & $y=0.4991x-0.6357$\\
$\beta=15.5$ &$I=[-0.5 ; 0.5]$   & $y=0.4945x+0.1765$\\
$\beta=20$ & $I=[ 0.4;0.6 ]$  & $y=0.5042x-0.7474$\\
\noalign{\smallskip}\hline\noalign{\smallskip}
\end{tabular}
\end{table}
One important issue in the numerical experiments is the approximation of the liming measures $\mu_{\beta}$ resp.\ their densities $p_{\beta}$.
For $\beta=1,2,4$ we use the MATLAB toolbox by Bornemann (c.f.~\cite{Bornemann2}) for a fast and precise evaluation of the related gap probabilities. Then we obtain the limiting densities $p_{\beta}$ by numerical differentiation.
For $\beta \in \mathbb R_+ \setminus \{1,2,4\}$ no such precise numerical schemes for the evaluation of $p_{\beta}$ are available. 
Instead, we use the generalised Wigner surmise (see \cite{LeClaer}), which is only an approximation to the limiting distribution. One may wonder how one can test numerically a limiting law without  knowing its precise form.  Looking at (\ref{report}) one notes that the numerics will not detect the replacement of $\int_{-\infty}^s d\mu_{\beta}$ by an approximation $\int_{-\infty}^s d\hat{\mu}_{\beta}$ as long as their deviation is small compared to $E_N$.
As it turns out, the Wigner surmise approximates the true limiting law well enough to confirm (\ref{report}) for the range of $N$ and  $\beta$ that we have tested. Moreover,  since  in all our experiments $E_N$ took values below $0.02$ we may safely infer that the difference between the distribution functions of the Wigner surmise and the true distribution is less than $0.01$ for all values of $\beta$ that we have investigated, i.e.~$\beta \in\{ 7, 15.5, 20\}$.

\section*{acknowledgement}
Both authors acknowledge support from the Deutsche Forschungsgemeinschaft in the framework of the SFB/TR 12 ``Symmetries and Universality in Mesoscopic Systems''. We are grateful to Peter Forrester for useful remarks.

\bibliographystyle{plain}
\bibliography{bibliography}

\end{document}